\documentclass[11pt]{amsart}

\usepackage[T1]{fontenc}
\usepackage[margin=1.12in]{geometry}
\usepackage{amsmath,amssymb,amsthm}
\usepackage{booktabs,array}
\usepackage{xcolor}
\usepackage{pgfplots}
\pgfplotsset{compat=1.18}
\usepackage{enumitem}
\usepackage{microtype}
\usepackage{tikz-cd}
\usepackage[hidelinks]{hyperref}

\newtheorem{theorem}{Theorem}[section]
\newtheorem{proposition}[theorem]{Proposition}
\newtheorem{lemma}[theorem]{Lemma}
\newtheorem{corollary}[theorem]{Corollary}
\theoremstyle{definition}
\newtheorem{problem}[theorem]{Problem}
\newtheorem{conjecture}[theorem]{Conjecture}
\theoremstyle{remark}
\newtheorem{remark}[theorem]{Remark}

\newcommand{\R}{\mathbb{R}}
\newcommand{\Z}{\mathbb{Z}}
\newcommand{\conv}{\operatorname{conv}}
\newcommand{\aff}{\operatorname{aff}}

\newcommand{\Vol}{\operatorname{Vol}}
\newcommand{\initial}{\operatorname{in}}
\newcommand{\cU}{\mathcal{U}}
\newcommand{\cF}{\mathcal{F}}
\newcommand{\cQ}{\mathcal{Q}}
\newcommand{\cN}{\mathcal{N}}
\newcommand{\cB}{\mathcal{B}}
\newcommand{\cC}{\mathcal{C}}
\newcommand{\cR}{\mathcal{R}}

\newcommand{\cH}{\mathcal{H}}
\newcommand{\cI}{\mathcal{I}}
\setlist[enumerate]{leftmargin=2.25em,itemsep=2pt,topsep=4pt}
\hypersetup{
  pdftitle={Most (0,1)-polytopes are not normal},
  pdfauthor={Santiago Morales},
  pdfsubject={Unimodular triangulations of (0,1)-polytopes},
  pdfkeywords={(0,1)-polytopes, normal polytopes, unimodular triangulations, flag triangulations, toric ideals, Koszul algebras, delta-matroids, random polytopes}
}

\title[Most $(0,1)$-polytopes are not normal]
{Most $(0,1)$-polytopes are not normal}

\author{Santiago Morales}
\address{Department of Mathematics, University of California, Davis}
\email{moralesduarte@ucdavis.edu}

\begin{document}
\raggedbottom

\begin{abstract}
{We prove that the proportion of $0/1$-equivalence classes of $d$-dimensional $(0,1)$-polytopes that are normal tends to zero at least at a double exponential rate as $d\to\infty$. As a consequence, the same holds for any of the following classes given by the type of triangulation possible: (a) quadratic, (b) flag unimodular, (c) regular unimodular, or (d) unimodular, among others. 

We classify the $0/1$-equivalence classes of $d$-dimensional $(0,1)$-polytopes for $d\leq5$ according to whether they admit a unimodular, flag unimodular, or quadratic triangulation. In dimension five, exactly $175$ out of $1{,}226{,}525$ classes have a flag unimodular triangulation, but no quadratic triangulation. Among them, there are polytopes whose toric rings are not Koszul; thus, we find the first polytopes that have a flag unimodular triangulation, but whose toric ring is not Koszul. In contrast with the matroid case, we exhibit a delta-matroid polytope that is not normal.}
\end{abstract}

\maketitle

\section{Introduction}\label{sec:introduction}

Unimodular triangulations have nice properties and a lot is known about when they exist \cite{HPPS21}. Let $P\subset\R^d$ be a lattice polytope. A triangulation of $P$ is called \emph{unimodular} if every maximal simplex has normalized volume one, and it is called \emph{flag} if every minimal nonface has two vertices.  A triangulation is called \emph{quadratic} if it is regular, unimodular, and flag. We call $P$ \emph{normal} if its associated toric ring $K[P]$ is a normal domain.

Writing $A=P\cap\Z^d$ and $\widetilde A=\{(1,a):a\in A\}$, we use the following hierarchy of triangulation and covering properties found in \cite[Section~1.2.5]{HPPS21}:

\begin{enumerate}[label=(\roman*)]
\item[($\cQ$)] $P$ has a quadratic triangulation;
\item[($\cF\cU$)] $P$ has a flag unimodular triangulation;
\item[($\cR\cU$)] $P$ has a regular unimodular triangulation;
\item[($\cU$)] $P$ has a unimodular triangulation;
\item[($\cB\cC$)] $P$ has a unimodular binary cover, that is, a cycle of
unimodular simplices generating $H_d(P,\partial P;\mathbb Z_2)$;
\item[($\cC$)] $P$ has a unimodular cover, that is, $P$ is the union of
unimodular simplices contained in $P$;
\item[($\cH$)] $P$ has the integer decomposition property and $\mathbb R_{\geq0}\widetilde A$ has a free Hilbert
cover, meaning that every point of $\mathbb R_{\geq0}\widetilde A\cap\mathbb Z^{d+1}$ is a nonnegative integral combination of linearly independent elements of
$\widetilde A$;
\item[($\cI$)] $P$ has the \emph{integer decomposition property} (IDP), meaning that, for every $k\geq1$, every point of $kP\cap\Z^d$ is a sum of $k$ points of $P\cap\Z^d$;
\item[($\cN$)] $P$ is normal.

\end{enumerate}
where
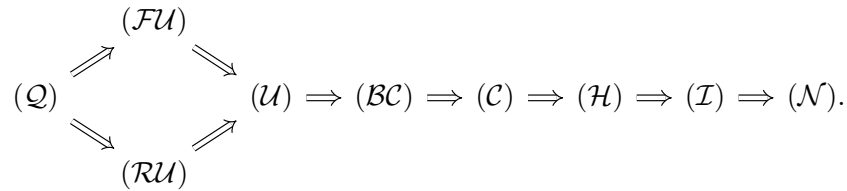
\[
\begin{tikzcd}[column sep=1.25em, row sep=small]
  & (\cF\cU) \arrow[dr, Rightarrow] \\
(\cQ) \arrow[ur, Rightarrow] \arrow[dr, Rightarrow]
  & & (\cU) \arrow[r, Rightarrow]
  & (\cB\cC) \arrow[r, Rightarrow]
  & (\cC) \arrow[r, Rightarrow]
  & (\cH) \arrow[r, Rightarrow]
  & (\cI) \arrow[r, Rightarrow]
  & (\cN). \\
  & (\cR\cU) \arrow[ur, Rightarrow]
\end{tikzcd}
\]

The relation between properties $(\cI)$ and $(\cN)$ is stated precisely in Lemma~\ref{lem:normal-idp}.

$(0,1)$-polytopes form a central class in polyhedral combinatorics, and their low-dimensional $0/1$-equivalence classes can be enumerated explicitly \cite{Zie00,Aic00,DLRS10}. The purpose of this paper is to classify the $0/1$-equivalence classes of $(0,1)$-polytopes in small dimensions according to properties $(\cQ)$, $(\cF\cU)$, and $(\cU)$, and to show that all of the properties in the hierarchy occur rarely among $0/1$-equivalence classes. We give a simple proof that the proportion of normal $0/1$-equivalence classes of $d$-dimensional $(0,1)$-polytopes tends to zero at a double exponential rate as $d$ grows. Therefore, the same holds for every property in the above hierarchy.

On the algebraic side, we consider a second, parallel hierarchy, applied to the toric ideals of $(0,1)$-polytopes. Let $S=K[x_1,\ldots,x_n]$ and let $I\subset S$ be a homogeneous ideal without linear forms. Then
\begin{equation}\label{eq:algebra-hierarchy}
 I\text{ has a quadratic Gr\"obner basis}
 \;\Longrightarrow\;
 S/I\text{ is Koszul}
 \;\Longrightarrow\;
 I\text{ is generated by quadrics}.
\end{equation}
The first implication is due to Anick \cite{Ani86}; see also \cite[Section~4]{Fro99} and \cite[Corollary~2.1.3]{BGT97}.  The second implication is standard. For a simplicial complex $\Delta$, Fr\"oberg's theorem says that its
Stanley--Reisner ring $K[\Delta]$ is Koszul if and only if $\Delta$ is flag \cite{Fro99}. 

The two parallel hierarchies are linked by Sturmfels' correspondence \cite{Stu96}: $P$ has a quadratic triangulation if and only if $I_P$ has a squarefree quadratic Gr\"obner basis. By Fr\"oberg's theorem, a quadratic triangulation forces $S/I_P$ to be Koszul.

Ohsugi and Hibi first constructed a normal $(0,1)$-polytope without regular unimodular triangulations \cite{OH99normal}, that is, whose toric ideal has no squarefree initial ideal. Ohsugi later constructed an infinite family with the same property \cite{Ohs02}. 

Backman and Liu proved that every matroid base polytope, a class of $(0,1)$-polytopes, has a regular unimodular triangulation \cite{BL25}. More recently,  it was shown that the base polytope of the Fano matroid has \textit{flag unimodular triangulations} but \textit{no quadratic triangulation}.  Thus, its toric ideal has \textit{no quadratic Gr\"obner basis}. This result was obtained by De Loera, Ferroni, Morales, and Rambau \cite{DLFMR26} and, independently, by Backman, Cheung, Laso\'n, Liu, and Micha{\l}ek \cite{BCLLM26}. It answered in the negative Herzog and Hibi's question whether the toric ideal of every (poly)matroid base polytope admits a quadratic Gr\"obner basis, a Gr\"obner-basis version of White's conjecture.

Herzog and Hibi also posed the question

\begin{problem}[Herzog--Hibi]\label{prob:matroid-koszul}
    \cite[p.~241]{HH02} Is every polymatroid base ring Koszul?
\end{problem}

De Loera, Ferroni, Morales, and Rambau found no matroid without a flag unimodular triangulation in their computations on ground sets of size eight and proposed the following conjecture.

\begin{conjecture}[De Loera--Ferroni--Morales--Rambau]\label{conj:polymatroid}
Every polymatroid base polytope has a flag unimodular triangulation
\cite[Conjecture~6.2]{DLFMR26}.
\end{conjecture}

The two questions are related, but we show here that they are not equivalent. Fr\"oberg's theorem applies to the Stanley--Reisner ring of a triangulation, and Sturmfels' correspondence identifies the Stanley--Reisner ideal of a \textit{regular} unimodular triangulation with an initial ideal of the toric ideal; a \textit{nonregular} triangulation gives no such identification.  The examples below show that this distinction is essential. 

{Following Ziegler \cite{Zie00}, two $d$-dimensional $(0,1)$-polytopes in $\R^d$ are $0/1$-equivalent if a symmetry of the $d$-cube maps one to the other. We enumerate all $0/1$-equivalence classes for $d\leq5$ using methods developed by Aichholzer \cite{Aic00}. We then classify all of the enumerated polytopes according to the $0/1$-equivalence-invariant properties $\cU$, $\cF\cU$, and $\cQ$ using the SAT-solver methods of \cite{DLFMR26}.}

\begin{theorem}[Census of small $(0,1)$-polytopes]\label{thm:main}
{For every $d\leq5$, we classify all $0/1$-equivalence classes of $d$-dimensional $(0,1)$-polytopes according to whether they admit a unimodular triangulation, a flag unimodular triangulation, or a quadratic triangulation; see Table~\ref{tab:census}. In particular:}
\begin{enumerate}[label=(\roman*)]
\item {For $d\leq4$, every $d$-dimensional $(0,1)$-polytope that has a flag unimodular triangulation also has a quadratic triangulation.}
\item {In dimension five, exactly $175$ $0/1$-equivalence classes of $5$-dimensional $(0,1)$-polytopes have a flag unimodular triangulation but no quadratic triangulation.}
\end{enumerate}
\end{theorem}

\begin{table}[ht]
\centering
\small
\begin{tabular}{@{}crrrrr@{}}
\toprule
$d$ & all classes & $\cQ$ & $\cF\cU\setminus\cQ$ & $\cU\setminus\cF\cU$ & not $\cU$\\
\midrule
$1$ & $1$ & $1$ & $0$ & $0$ & $0$\\
$2$ & $2$ & $2$ & $0$ & $0$ & $0$\\
$3$ & $12$ & $10$ & $0$ & $1$ & $1$\\
$4$ & $347$ & $212$ & $0$ & $91$ & $44$\\
$5$ & $1{,}226{,}525$ & $153{,}791$ & \textbf{175} & $661{,}337$ & $411{,}222$\\
\bottomrule
\end{tabular}
\caption{{$0/1$-equivalence classes of $d$-dimensional $(0,1)$-polytopes according to properties $\cU$: existence of a unimodular triangulation; $\cF\cU$: existence of a flag unimodular triangulation; and $\cQ$: existence of a quadratic triangulation.}}
\label{tab:census}
\end{table}

Within the family $\cF\cU\setminus\cQ$, we find polytopes whose toric rings are not Koszul. We also enumerate all flag unimodular triangulations of the polytopes in $\cF\cU\setminus\cQ$. All of these triangulations are non-regular; see Figure~\ref{fig:flagcensus}.

\begin{proposition}\label{prop:cubic-example}
There exists a $(0,1)$-polytope with a flag unimodular triangulation but whose toric ideal is not generated by quadrics.
\end{proposition}

\begin{proposition}\label{prop:quadric-example}
There exists a $(0,1)$-polytope with a flag unimodular triangulation such that its toric ideal is generated by quadrics but its toric ring is  not Koszul.
\end{proposition}

Taking lattice pyramids extends these examples to every higher dimension.

\begin{corollary}[Pyramidal extensions]\label{cor:pyramids}
{For every $d\geq5$, there exists a $d$-dimensional $(0,1)$-polytope that has a flag unimodular triangulation but a non-Koszul toric ring.} Moreover, the toric ideal can be chosen to be generated by quadrics.
\end{corollary}

Ohsugi and Hibi gave non-Koszul toric rings of $(0,1)$-polytopes whose toric ideals are generated by quadrics \cite[Example~2.1]{OH99quadratic}, and Matsuda later gave Gorenstein examples \cite{Mat18}. We verified that those examples do not have flag unimodular triangulations. The polytopes in Corollary~\ref{cor:pyramids} are therefore the \textit{first} lattice polytopes that have flag unimodular triangulations and non-Koszul toric rings.

We also consider \emph{delta-matroid} polytopes. A \emph{delta-matroid} generalizes a matroid by weakening basis exchange and allowing feasible sets of different cardinalities \cite{EFLS24,FMN18}. More precisely, a delta-matroid on $[n]$ is a pair $D=([n],\mathcal F)$, where $\varnothing\neq\mathcal F\subseteq2^{[n]}$ and, for all $X,Y\in\mathcal F$ and every $u\in X\mathbin\triangle Y$, there exists $v\in X\mathbin\triangle Y$, possibly $v=u$, such that $X\mathbin\triangle\{u,v\}\in\mathcal F$. Its delta-matroid polytope is
\[
 P(D)=\conv\{\mathbf 1_F:F\in\mathcal F\}\subset[0,1]^n.
\]
Here we show that Backman and Liu's result for matroid base polytopes does not extend to delta-matroids \cite{BL25}.

\begin{proposition}[{A nonnormal delta-matroid polytope}]\label{prop:delta}
{Let $D_\star=([4],\mathcal F_\star)$, where
\[
 \mathcal F_\star=\{\varnothing,\{4\}\}\cup\binom{[4]}2\cup\{[4]\}.
\]
Then $D_\star$ is a delta-matroid. Its delta-matroid polytope is}
\[
\begin{aligned}
 P_\star=\conv\{&0000,0001,1100,1010,1001,\\
                  &0110,0101,0011,1111\},
\end{aligned}
\]
and does not have the integer decomposition property. Its vertices affinely generate $\Z^4$, so $P_\star$ is also not normal. Thus, it has no unimodular triangulation.
\end{proposition}

A proof of Proposition~\ref{prop:delta} appears in Section~\ref{sec:probability}. 

The classification also motivates an asymptotic question. The positive existence results surveyed in \cite{HPPS21}, and those for structured families such as matroid base polytopes \cite{BL25}, contrast sharply with Table~\ref{tab:census}.  {Haase, Paffenholz, Piechnik, and Santos formulated an asymptotic question for fixed dimension as normalized volume grows \cite[Section~1.5]{HPPS21} (see Section~\ref{sec:open}). Here, instead, the dimension $(0,1)$-polytopes grows.

\begin{corollary}[Normality is asymptotically rare]\label{cor:classes-normal}
{For each $d$, let $N=N(d)$ denote the number of $0/1$-equivalence classes of $d$-dimensional normal $(0,1)$-polytopes, and let $N_d$ denote the total number of $0/1$-equivalence classes of $d$-dimensional $(0,1)$-polytopes. As $d\to\infty$,}
\[
 {\frac{N}{N_d}\leq(1+o(1))2^d d!\,e^{-2^{d}}\longrightarrow0.}
\]
\end{corollary}

\begin{corollary}[{Triangulation and covering properties are asymptotically rare}]\label{cor:classes-triangulations}
{Denote by $\mathrm U_d$, $\mathrm{FU}_d$, $\mathrm{RU}_d$, $\mathrm{BC}_d$, $\mathrm C_d$, and $\mathrm Q_d$ the numbers of $0/1$-equivalence classes of $d$-dimensional $(0,1)$-polytopes that have, respectively, a unimodular triangulation, a flag unimodular triangulation, a regular unimodular triangulation, a unimodular binary cover, a unimodular cover, and a quadratic triangulation. Then}
\[
 {\frac{\mathrm U_d}{N_d},\quad
 \frac{\mathrm{FU}_d}{N_d},\quad
 \frac{\mathrm{RU}_d}{N_d},\quad
 \frac{\mathrm{BC}_d}{N_d},\quad
 \frac{\mathrm C_d}{N_d},\quad
 \frac{\mathrm Q_d}{N_d}\leq(1+o(1))2^d d!\,e^{-2^{d}}\longrightarrow0 
 \qquad\text{as }d\to\infty.}
\]

\end{corollary}

This follows immediately from Corollary~\ref{cor:classes-normal}, since
every listed property implies $(\cN)$. Section~\ref{sec:probability} proves Corollary~\ref{cor:classes-normal} through the Bernoulli model for $(0,1)-$polytopes.

\section{A census of triangulations for \texorpdfstring{$(0,1)$}{(0,1)}-polytopes of small dimension}
\label{sec:census}

\subsection{Toric rings and triangulations}

Let $P=\conv(A)\subset\R^d$ be a lattice polytope, where $A=P\cap\Z^d$. Set
\[
 S_A=K[x_a:a\in A]
\]
and consider the homomorphism
\[
 \varphi_A:S_A\longrightarrow K[t,z_1^{\pm1},\ldots,z_d^{\pm1}],
 \qquad x_a\longmapsto tz^a.
\]
The toric ideal and the toric ring of $P$ are
\[
 I_P=\ker(\varphi_A),\qquad K[P]=S_A/I_P.
\]

Write $\widetilde A=\{(1,a):a\in A\}$ and set
\[
 M_P=\mathbb N\widetilde A,\qquad
 G_P=\mathbb Z\widetilde A,\qquad
 \sigma_P=\mathbb R_{\geq0}\widetilde A.
\]
A full-dimensional lattice polytope is called \emph{spanning} if its lattice
points affinely generate $\Z^d$, or equivalently, if
$G_P=\mathbb Z^{d+1}$.

\begin{lemma}[Normality, integral closure, and integer decomposition, see \cite{BH98}]\label{lem:normal-idp}
Let $P=\conv(A)\subset\R^d$ be a lattice polytope with
$A=P\cap\Z^d$, and use the notation above.
\begin{enumerate}[label=(\roman*)]
\item The homomorphism $\varphi_A$ identifies $K[P]$ with the affine semigroup
ring $K[M_P]$. The following conditions are equivalent:
\begin{enumerate}[label=(\alph*)]
\item $P$ is normal, that is, $K[P]$ is a normal domain;
\item $K[P]$ is integrally closed in its field of fractions;
\item $M_P$ is saturated in its group $G_P$, that is,
\[
 M_P=\sigma_P\cap G_P.
\]
\end{enumerate}
\item The following conditions are equivalent:
\begin{enumerate}[label=(\alph*)]
\item $P$ has the integer decomposition property;
\item \(M_P=\sigma_P\cap\mathbb Z^{d+1};\)
\item for every $k\geq1$, every point of $kP\cap\Z^d$ is a sum of $k$
points of $P\cap\Z^d$.
\end{enumerate}
\item The integer decomposition property implies normality. If $P$ is
spanning, then normality and the integer decomposition property are equivalent.
\end{enumerate}
\end{lemma}

\begin{remark}
    In some sources, a polytope satisfying (ii) is called \emph{integrally closed} see \cite[Remark~3.7]{CHHH14}.
\end{remark}

\begin{proof}
The image of $\varphi_A$ is $K[M_P]\subset
K[t,z_1^{\pm1},\ldots,z_d^{\pm1}]$, which proves the first assertion in (i)
and shows that $K[P]$ is a domain. A normal domain is, by definition, a domain
integrally closed in its field of fractions. For an affine semigroup $M$, the
semigroup ring $K[M]$ is normal if and only if $M$ is normal
\cite[Theorem~6.1.4]{BH98}; moreover, the normalization of $K[M_P]$ is
$K[\sigma_P\cap G_P]$ \cite[Proposition~1.1.2]{BGT97}. Thus the criterion in
(i) is exactly the normality criterion in \cite[Proposition~13.5]{Stu96}.

For $k\geq1$, a point $(k,x)$ lies in $\sigma_P$ if and only if
$x\in kP$. Since every generator $(1,a)$ of $M_P$ has first coordinate one,
$(k,x)\in M_P$ if and only if $x=a_1+\cdots+a_k$ for some
$a_1,\ldots,a_k\in A$. This proves (ii).

Finally, $M_P\subseteq\sigma_P\cap G_P\subseteq
\sigma_P\cap\mathbb Z^{d+1}$. Hence the equality in (ii) implies the equality
in (i). If $P$ is spanning, then $G_P=\mathbb Z^{d+1}$, so the two equalities
coincide. This proves (iii).
\end{proof}

We call $I_P$ \emph{G-quadratic} if it has a quadratic Gr\"obner basis on the
variables $x_a$, $a\in A$.  If $\Delta$ is a simplicial complex on $A$, then
$I_\Delta\subset S_A$ denotes its Stanley--Reisner ideal and
$K[\Delta]=S_A/I_\Delta$ its Stanley--Reisner ring.

A $(0,1)$-polytope is a lattice polytope $P=\conv(A)\subset[0,1]^d$, where $A\subset\{0,1\}^d$. A $(0,1)$-polytope has no lattice points other than its vertices.

\begin{lemma}\label{lem:lattice-points}
If $P=\conv(A)\subset[0,1]^d$, where $A\subset\{0,1\}^d$, then
$P\cap\Z^d=A$.
\end{lemma}

\begin{proof}
Every lattice point of $[0,1]^d$ is a cube vertex.  For $v\in\{0,1\}^d$, the affine
function
\[
 \ell_v(x)=\sum_{v_i=1}x_i+\sum_{v_i=0}(1-x_i)
\]
has the unique maximum $d$ at $v$.  Hence $v$ is not a convex combination of the other
cube vertices.
\end{proof}

Every polytope with a unimodular triangulation is spanning. The following standard dictionary combines results of Fr\"oberg and Sturmfels; see also \cite[Section~9.4]{DLRS10}.

\begin{proposition}\label{prop:dictionary}
Let $P=\conv(A)\subset\R^d$ be a lattice polytope with $A=P\cap\Z^d$, and let $\Delta$ be a triangulation of $P$.
\begin{enumerate}[label=(\roman*)]
\item $K[\Delta]$ is Koszul if and only if $\Delta$ is flag.
\item If $\Delta$ is regular and is induced by a weight vector $\omega$, then
\[
 \sqrt{\initial_\omega(I_P)}=I_\Delta.
\]
If, in addition, $\Delta$ is unimodular with respect to the affine lattice generated by $A$, then
$\initial_\omega(I_P)=I_\Delta$.
\item If $\Delta$ is quadratic, then $I_P$ has a quadratic Gr\"obner basis and $K[P]$
is Koszul.
\item If $P$ is a spanning $(0,1)$-polytope, then $P$ has a quadratic
triangulation if and only if $I_P$ has a quadratic Gr\"obner basis.
\end{enumerate}
\end{proposition}

\begin{proof}
Part (i) is Fr\"oberg's theorem.  Part (ii) is the Gr\"obner--triangulation
correspondence {\cite[Chapter~8]{Stu96}; see also \cite[Section~9.4]{DLRS10}.}  If $\Delta$ is regular, unimodular, and flag,
then $\initial_\omega(I_P)=I_\Delta$ is generated by squarefree quadratic monomials.
This proves (iii) by \eqref{eq:algebra-hierarchy}.

Suppose now that $I_P$ has a quadratic Gr\"obner basis.  A quadratic minimal generator of its initial ideal cannot be a square.  Indeed, a binomial with initial monomial $x_a^2$ would give $2a=b+c$.  Since $a,b,c$ are $(0,1)$-vectors, this equality forces $a=b=c$.  Hence the initial ideal is squarefree and defines a regular flag triangulation.
It is unimodular with respect to the affine lattice generated by $A$.  Since $P$ is spanning, it is
unimodular in $\Z^d$.
\end{proof}

The preceding proposition explains the role of regularity. If $\Delta$ is any
unimodular triangulation of a $d$-dimensional lattice polytope $P$, then $P$ has the integer decomposition property and hence is normal by Lemma~\ref{lem:normal-idp}, and
\begin{equation}\label{eq:hilbert-series}
 H_{K[P]}(t)=H_{K[\Delta]}(t)
 =\frac{h_\Delta(t)}{(1-t)^{d+1}}
 =\frac{h_P^*(t)}{(1-t)^{d+1}}.
\end{equation}
See \cite{Sta80,Stu96}.  If $\Delta$ is flag, then $K[\Delta]$ is Koszul.  This implies that $K[P]$ is Koszul
when $\Delta$ is regular, because $I_\Delta$ is then an initial ideal of $I_P$.  For a
nonregular triangulation, \eqref{eq:hilbert-series} is only a numerical identity.  The
Hilbert series does not determine the graded Betti numbers of the residue field.  The polytope in Proposition~\ref{prop:P1-koszul} below gives an explicit example.

\begin{proposition}[Face heredity]\label{prop:faces}
Let $P=\conv(A)\subset\R^d$ be a lattice polytope, let $\Delta$ be a triangulation of $P$, and let $F$ be a face of $P$.  Write $\Delta|_F=\{\sigma\in\Delta:\sigma\subseteq F\}$.
\begin{enumerate}[label=(\roman*)]
\item The complex $\Delta|_F$ is an induced triangulation of $F$.  If $\Delta$ is
regular, unimodular, or flag, then $\Delta|_F$ has the same property.  Unimodularity is
measured in $\aff(F)\cap\Z^d$.
\item If $I_P$ is generated by quadrics, then $I_F$ is generated by quadrics; if $I_P$ is G-quadratic, then $I_F$ is G-quadratic.  If $K[P]$ is Koszul, then $K[F]$ is Koszul. If $P$ has the integer decomposition property, then $F$ has the integer decomposition property; if $P$ is normal, then $F$ is normal.
\end{enumerate}
\end{proposition}

\begin{proof}
For (i), write $F=P\cap\{\ell=c\}$, where $\ell\leq c$ on $P$.  Intersecting the
simplices of $\Delta$ with $\{\ell=c\}$ gives an induced triangulation of $F$.  A
lifting function restricts to $F$, an induced subcomplex of a flag complex is flag, and a
face of a unimodular simplex is unimodular in its affine lattice.

For (ii), after multiplying $\ell$ by a positive integer, we may suppose that $\ell$ and
$c$ are integral.  Replacing the generator $tz^a$ of $K[P]$ by
$tz^a s^{c-\ell(a)}$ gives an isomorphic homogeneous semigroup ring.  In this
realization, $K[F]$ is the combinatorial pure subring generated by the monomials not
involving $s$.  The assertions concerning quadratic generation, G-quadraticity, and Koszulness follow from
\cite[Propositions~1.1 and~1.3, and Corollary~2.5]{OHH00}.  For the final two assertions, let
$\lambda(1,a)=c-\ell(a)$ on the homogenized semigroup. This functional is
nonnegative on the generators and vanishes exactly on those belonging to
$F$. If $u$ lies in the cone of the face and in $\mathbb Z^{d+1}$, the
integer decomposition property of $P$ gives a representation of $u$ as a nonnegative integral
sum of generators of $P$. If instead $u$ lies in the cone and group of the
face semigroup, normality of $P$ gives the same conclusion. In either case,
applying $\lambda$ forces every summand to belong to $F$, so $F$ has the
corresponding property.
\end{proof}

\subsection{\texorpdfstring{{Enumerating $0/1$-equivalence classes}}{Enumerating 0/1-equivalence classes}}\label{subsec:methodology}
We follow Aichholzer's enumeration strategy \cite{Aic00}.
Let
\[
 G_d=(\mathbb Z/2\mathbb Z)^d\rtimes\mathfrak S_d
\]
{be the symmetry group of the $d$-cube; its orbits are the $0/1$-equivalence classes.} A subset of $\{0,1\}^d$ is stored as a bitmask,
and its canonical representative is the least bitmask in its $G_d$-orbit. Starting with the canonical representatives having $m$ vertices, one adjoins one cube vertex in every possible way, canonicalizes the
result, and keeps each canonical representative once.  We check full-dimensionality by fraction-free elimination on a matrix of difference vectors.

We check completeness of this enumeration independently at every value of $m$.  If
$g\in G_d$ has cycles $c$ on the $2^d$ cube vertices, then an $m$-element subset fixed by
$g$ is a union of such cycles.  By Burnside's lemma \cite[Chapter~10]{vLW01},
\[
 \#\{\text{$G_d$-orbits of $m$-element subsets}\}
 =\frac{1}{|G_d|}\sum_{g\in G_d}[z^m]
   \prod_{c\in\operatorname{cyc}(g)}(1+z^{|c|}).
\]
We enumerated the $0/1$-equivalence classes of $d$-dimensional polytopes $P=\conv(A)\subset[0,1]^d$. The
numbers of classes are}
\[
 1,\quad2,\quad12,\quad347,\quad1{,}226{,}525
 \qquad(d=1,\ldots,5).
\]
These counts agree with \cite{Aic00,CG14}; the classification by triangulation type is summarized in Table~\ref{tab:census}.

Thus $\cF\cU\setminus\cQ$ first occurs in dimension five.  The $175$ exceptional classes have between $13$ and $23$ vertices and normalized volume between $29$ and $99$. They form $131$ affine-unimodular equivalence classes.  Each has a regular unimodular triangulation, and each proper face has a quadratic triangulation. 

\textbf{The flag encoding.} Fix $P=\conv(A)$.  We find flag unimodular triangulations using the carrier
formulation of \cite{DLFMR26}.  If $\Delta$ is unimodular, then for every $k\geq2$ and every lattice point $b\in kP\cap\Z^d$, there is exactly one multiset $\{a_1,\ldots,a_k\}\subseteq A$ with
$a_1+\cdots+a_k=b$ whose support is a face of $\Delta$.  For $k=2$, group the unordered pairs of distinct vertices by their sum:
\[
 \mathcal E_b=\bigl\{\{a,a'\}\subseteq A:a\neq a',\ a+a'=b\bigr\}.
\]

Each candidate graph is checked combinatorially.  It must have no clique with more than $d+1$ vertices; every clique must be contained in a clique of size $d+1$; and every such maximal clique must span a $d$-simplex of normalized volume one. There must be exactly $\Vol(P)$ maximal simplices.  We check that these simplices meet face-to-face and cover $P$ with the wall criterion of \cite{DHSS96}; see also
\cite[Section~8.2]{DLRS10}. 

\textbf{Enumeration of flag unimodular triangulations.}
For a fixed polytope, an accepted assignment determines exactly one flag unimodular triangulation. We enumerate all of them by a standard \textsc{AllSAT} procedure \cite{TS16}: after one triangulation is accepted, a blocking clause that excludes that specific triangulation is added, and the search is restarted.  When the instance becomes unsatisfiable, the list of all flag unimodular triangulations is complete.  As checks, the procedure finds no flag unimodular triangulation for the polytope $H=\conv\{000,100,010,001,111\}$, and it finds the expected $64$ flag unimodular triangulations of the $3$-cube \cite{DLRS10}.

Applied to the $175$ classes in $\cF\cU\setminus\cQ$, this gives exactly $11{,}529$ flag unimodular triangulations.  Five classes have exactly one flag unimodular triangulation, the median class has seven, and the largest count, belonging to a seventeen-vertex polytope of normalized volume $47$, is $2{,}014$.  Most classes admit only a handful of flag unimodular triangulations, while a few classes account for most of the total. Every one of the $11{,}529$ triangulations is nonregular. We show the resulting distribution in Figure~\ref{fig:flagcensus}.

\begin{figure}[ht]
\centering
\begin{tikzpicture}
\begin{axis}[
    ybar,
    bar width=15pt,
    width=0.94\textwidth,
    height=0.37\textwidth,
    ymin=0, ymax=47,
    ytick={0,10,20,30,40},
    symbolic x coords={$1$,$2$,$3$--$4$,$5$--$8$,$9$--$16$,$17$--$32$,$33$--$64$,
        $65$--$128$,$129$--$256$,$257$--$512$,$513$--$1024$,$1025$--$2048$},
    xtick=data,
    x tick label style={font=\footnotesize,rotate=40,anchor=north east,
        yshift=2pt,xshift=2pt},
    y tick label style={font=\footnotesize},
    ylabel={number of classes},
    ylabel style={font=\small},
    xlabel={flag unimodular triangulations},
    xlabel style={font=\small},
    nodes near coords,
    every node near coord/.append style={font=\footnotesize},
    axis lines*=left,
    ymajorgrids,
    major grid style={draw=gray!25},
    clip=false
]
\addplot[fill=gray!35,draw=black!75] coordinates {
    ($1$,5) ($2$,40) ($3$--$4$,20) ($5$--$8$,31) ($9$--$16$,27) ($17$--$32$,17)
    ($33$--$64$,15) ($65$--$128$,6) ($129$--$256$,3) ($257$--$512$,7)
    ($513$--$1024$,0) ($1025$--$2048$,4)
};
\end{axis}
\end{tikzpicture}
\caption{An enumeration of the flag unimodular triangulations of the $175$ classes in $\cF\cU\setminus\cQ$, binned by powers of two.  Five classes have exactly one flag unimodular triangulation; the largest count is $2{,}014$ triangulations.}
\label{fig:flagcensus}
\end{figure}
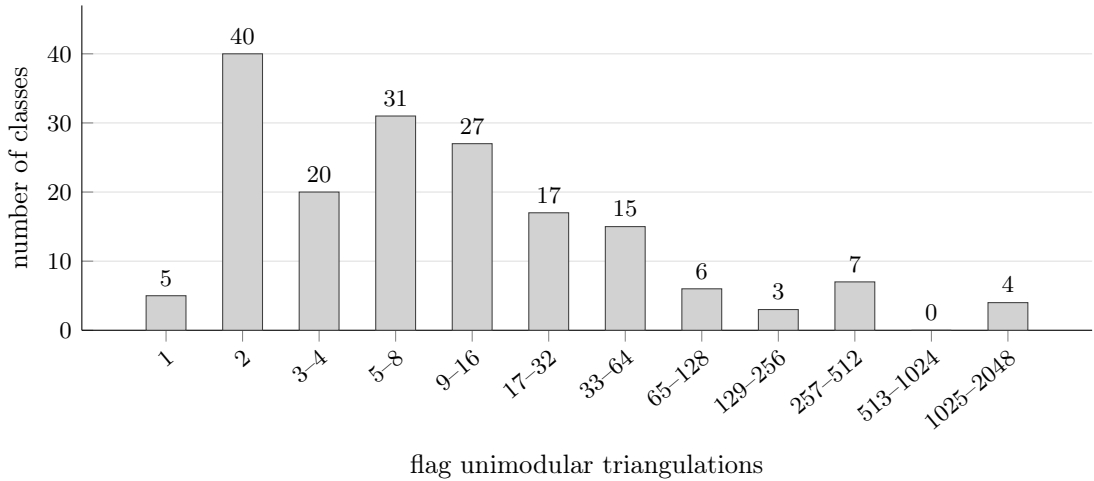

For each of the exceptional classes, we compute the toric ideal in \textsc{Macaulay2} \cite{M2} as the kernel of the monomial map $x_a\mapsto tz^a$.  The degrees of a minimal generating set decide quadratic generation: $160$ ideals are generated by quadrics, but $15$ require one cubic.  For the $160$ quadratically generated rings, a minimal resolution of the residue field is computed through homological degree four.  A nonzero graded Betti number off the linear strand
means that the toric ring is not Koszul.  Using this procedure, we find ten additional non-Koszul rings. The Koszulness of the remaining $150$ cases open.

The implementation uses Python with exact integer verification, PicoSAT and CaDiCaL for the Boolean searches \cite{Bie08,BFFH20}, and linear programming only to check regularity. 

\subsection{A small example}

For a monomial $x^u\in S_A$, put
\[
 \deg_A(x^u)=\sum_{a\in A}u_a(1,a).
\]
The fiber of an $A$-degree $b$ is the set of monomials having degree $b$.  Equivalently,
two monomials lie in the same fiber precisely when they have the same image under
$\varphi_A$; a binomial in $I_P$ is the difference of two monomials in one fiber.  For an
$A$-degree $b$, let $G_b^{(r)}$ be the graph whose vertices are the monomials in this
fiber.  Two vertices are adjacent if their difference is a monomial multiple of a
binomial in $I_P$ of degree at most $r$.

\begin{lemma}[Fiber connectivity]\label{lem:fibers}
Let $P=\conv(A)$ be a lattice polytope, and let $r\geq1$.  For each $A$-degree $b$, let $G_b^{(r)}$ be the graph whose vertices are the monomials of $A$-degree $b$, with two vertices adjacent when their difference is a monomial multiple of a binomial in $I_P$ of degree at most $r$.  Then $I_P$ is generated by binomials of degree at most $r$ if and only if $G_b^{(r)}$ is connected for every $A$-degree $b$.
\end{lemma}

\begin{proof}
Take a path in $G_b^{(r)}$, the difference of its endpoints is a sum of monomial
multiples of binomials of degree at most $r$.  Conversely, an expression in such
binomials gives a sequence of replacements inside the fiber. Hence the endpoints lie in
the same connected component.
\end{proof}

Consider the $(0,1)$-polytope
\begin{align*}
P_1=\conv\{&00000,10000,01000,10100,01100,10010,01010,11110,\\
            &00001,10001,01001,00101,00011\}.
\end{align*}
Its normalized volume is $29$, and
\[
 h_{P_1}^*(t)=1+7t+13t^2+7t^3+t^4.
\]
A \emph{quadratic fiber} (a sum of two vertices of $P_1$ that potentially form an edge of the triangulation) is indexed by a vector $z\in\Z^5$.  Because $P_1$ is a
$(0,1)$-polytope, every nontrivial quadratic fiber consists of squarefree monomials,
which we identify with unordered pairs of distinct vertices whose sum is $z$.  In a
flag unimodular triangulation exactly one pair in each nontrivial quadratic fiber is
selected as an edge.  We write a pair $\{a,b\}$ as $a+b$.

There are fourteen nontrivial quadratic fibers for $P_1$.  Seven of their selected edges
are forced.  In the following array, the first columns give the fiber sum $z$,
and the adjacent entry gives the edge which must be selected:
\[
\begin{array}{c|c@{\qquad}c|c}
01001&01000+00001&10001&10000+00001\\
01011&01010+00001&10011&10010+00001\\
01101&01100+00001&10101&10100+00001\\
11110&00000+11110&&
\end{array}
\]
For example, the fiber over $01001$ consists of
$00000+01001$ and $01000+00001$, and the second pair is forced. In the fiber over $11110$,
the selected pair is $00000+11110$ and the two rivals are $10100+01010$ and
$01100+10010$.

The next six fibers each have the two displayed possibilities.  The degree-three fiber constraints couple these choices: they cannot be selected independently.  Either all six
edges in column $\mathsf A$ occur, or all six edges in column $\mathsf B$ occur:
\[
\begin{array}{c|c|c}
\toprule
z&\mathsf A&\mathsf B\\
\midrule
01111&01010+00101&01100+00011\\
10111&10100+00011&10010+00101\\
11010&01000+10010&10000+01010\\
11011&10010+01001&01010+10001\\
11100&10000+01100&01000+10100\\
11101&01100+10001&10100+01001\\
\bottomrule
\end{array}
\]
The remaining quadratic fiber is independent of these two patterns:
\[
 11001:\qquad10000+01001\quad\text{or}\quad01000+10001.
\]
Thus the two global patterns and the two choices in the free fiber give four flag
unimodular triangulations.  Each has $29$ maximal simplices.

\begin{proposition}\label{prop:P1-regularity}
No flag unimodular triangulation of $P_1$ is regular.
\end{proposition}

\begin{proof}
For pattern $\mathsf A$, an inducing height function $\omega$ would satisfy
\begin{align*}
\omega(01100)+\omega(10001)&<\omega(10100)+\omega(01001),\\
\omega(10100)+\omega(00011)&<\omega(10010)+\omega(00101),\\
\omega(01010)+\omega(00101)&<\omega(01100)+\omega(00011),\\
\omega(10010)+\omega(01001)&<\omega(01010)+\omega(10001).
\end{align*}
The sum of the left-hand sides equals the sum of the right-hand sides.  This is
impossible.  Pattern $\mathsf B$ reverses all four inequalities and gives the same
contradiction.  The free fiber does not occur in the sum.
\end{proof}

The same four fibers show that there is a minimal cubic generator.  The fiber of $A$-degree
$(3;1,1,1,1,2)$ consists of
\begin{align*}
m_0&=x_{11110}x_{00001}^2,\\
m_1&=x_{10100}x_{01001}x_{00011},\\
m_2&=x_{01100}x_{10001}x_{00011},\\
m_3&=x_{10010}x_{01001}x_{00101},\\
m_4&=x_{01010}x_{10001}x_{00101}.
\end{align*}
The last four monomials form the quadratic-move cycle
\[
 m_1\longleftrightarrow m_2\longleftrightarrow m_4
 \longleftrightarrow m_3\longleftrightarrow m_1,
\]
whereas $m_0$ is isolated.

\begin{proposition}\label{prop:P1-koszul}
For every field $K$, the toric ideal $I_{P_1}$ is not generated by quadrics and the
toric ring $K[P_1]$ is not Koszul.  More precisely, $I_{P_1}$ has a minimal cubic
generator
\[
 x_{11110}x_{00001}^2-x_{10100}x_{01001}x_{00011}.
\]
\end{proposition}

\begin{proof}
The graph $G_{(3;1,1,1,1,2)}^{(2)}$ is disconnected.  Lemma~\ref{lem:fibers} therefore
shows that $I_{P_1}$ is not generated by quadrics.  The contrapositive of the second implication in
\eqref{eq:algebra-hierarchy} shows that $K[P_1]$ is not Koszul.
\end{proof}

Together with the four flag unimodular triangulations constructed above, Proposition~\ref{prop:P1-koszul} proves Proposition~\ref{prop:cubic-example}.

Let $\Delta$ be any one of the four flag unimodular triangulations of $P_1$.  Then
Fr\"oberg's theorem and \eqref{eq:hilbert-series} give
\[
 K[\Delta]\text{ is Koszul},\qquad K[P_1]\text{ is not Koszul},
\]
and
\[
 H_{K[\Delta]}(t)=H_{K[P_1]}(t)
 =\frac{1+7t+13t^2+7t^3+t^4}{(1-t)^6}.
\]
This explains why flagness does not imply Koszulness of the toric ring.  Flagness is a
statement about $I_\Delta$.  Only regularity identifies $I_\Delta$ with an initial ideal
of $I_P$.

Exactly $15$ exceptional ideals require one cubic in addition to quadrics.  The other
$160$ are generated by quadrics.  None has a quadratic Gr\"obner basis.  Ten of the corresponding rings satisfy
\[
 \beta_{3,4}^{K[P]}(K)\neq0
\]
and hence are not Koszul.  For the remaining $150$ rings, the resolution of the residue field is linear through homological degree four. Their Koszulness remains open.  All $175$ polytopes satisfy the clique-face condition of Mori and Ohsugi \cite{MO26}, which is necessary for quadratic generation.

The ten non-Koszul rings among the $160$ quadratically generated cases give Proposition~\ref{prop:quadric-example}.  We now prove the pyramidal extension stated in the introduction.

\begin{proof}[Proof of Corollary~\ref{cor:pyramids}]
Take iterated lattice pyramids.  If $P\subset\R^r$ is a $(0,1)$-polytope, then
\[
 \operatorname{pyr}(P)=\conv\bigl((P,0)\cup\{e_{r+1}\}\bigr)
\]
is again a $(0,1)$-polytope.  Coning preserves flagness and unimodularity, while a
quadratic triangulation of the pyramid restricts to a quadratic triangulation of its base.
Moreover,
$K[\operatorname{pyr}(P)]\simeq K[P][y]$.  Polynomial extension preserves quadratic
generation and preserves Koszulness in both directions.  Apply this to $P_1$, or to one
of the ten quadratically generated non-Koszul examples.
\end{proof}

\section{A probabilistic proof that normality is rare}
\label{sec:probability}

Let $S_d(p)\subseteq\{0,1\}^d$ be obtained by choosing each cube vertex independently
with probability $p$, and set $P_d(p)=\conv(S_d(p))$.  We use the standard
independent block argument from the probabilistic method; see \cite[Chapter~1]{AS16}.

\begin{proof}[Proof of Proposition~\ref{prop:delta}]
The even subsets of $[4]$ form an even delta-matroid, so only exchanges
involving the additional feasible set $\{4\}$ require verification. Suppose
first that $X=\{4\}$ and $Y$ is even. For any
$u\in X\mathbin\triangle Y$, take $v=u$. Then
$X\mathbin\triangle\{u\}$ is either $\varnothing$ or a two-element set,
hence feasible. Conversely, suppose that $X$ is even and $Y=\{4\}$. If
$|X\mathbin\triangle Y|\geq2$, choose $v\neq u$ in
$X\mathbin\triangle Y$; toggling two distinct elements preserves even
cardinality. If $|X\mathbin\triangle Y|=1$, take $v=u$, which gives $Y$.
Thus $D_\star$ is a delta-matroid.

Since $0000$ is a vertex, the differences
\[
 0001-0000=e_4,\qquad
 1001-0001=e_1,\qquad
 0101-0001=e_2,\qquad
 0011-0001=e_3
\]
show that the vertices of $P_\star$ affinely generate $\Z^4$. Moreover,
\[
 \frac12(1110)=\frac14(0000+1100+1010+0110)\in P_\star,
\]
so $1110\in2P_\star\cap\Z^4$. By Lemma~\ref{lem:lattice-points}, the
lattice points of $P_\star$ are exactly its vertices. The point
$1110$ is not a sum of two of them: its last coordinate forces both summands
to have last coordinate zero, and their supports would then be disjoint
even-cardinality subsets whose union is $\{1,2,3\}$, which is impossible.
Hence $P_\star$ does not have the integer decomposition property.
Since $P_\star$ is spanning, Lemma~\ref{lem:normal-idp}(iii) shows that
$P_\star$ is not normal. Finally,
$(\cC)\Rightarrow(\cI)$ shows that $P_\star$ has no unimodular cover, and therefore no unimodular triangulation.
\end{proof}

\begin{theorem}[Normality in the Bernoulli model]\label{thm:bernoulli-normality}
For every $d\geq4$ and $p\in(0,1)$,
\[
 \Pr\bigl[P_d(p)\text{ is normal}\bigr]
 \leq\left(1-p^9(1-p)^7\right)^{2^{d-4}}
 \leq\exp\!\left(-2^{d-4}p^9(1-p)^7\right).
\]
\end{theorem}

\begin{proof}
Partition the $d$-cube into the $2^{d-4}$ disjoint coordinate $4$-cubes obtained
by fixing the last $d-4$ coordinates.  In each block, the probability that the
chosen vertices are exactly a copy of the vertex set of $P_\star$ is
$p^9(1-p)^7$, and these events are independent.  If one occurs, the convex hull
of the chosen vertices in that block is a face of $P_d(p)$ lattice-isomorphic
to $P_\star$.  Propositions~\ref{prop:delta} and~\ref{prop:faces} therefore imply
that $P_d(p)$ is not normal.  The first inequality follows by independence, and
the second from $1-x\leq e^{-x}$.
\end{proof}

The proof uses no special feature of $P_\star$; any fixed nonnormal $(0,1)$-polytope can be used in its place. At $p=1/2$, Theorem~\ref{thm:bernoulli-normality} gives
\[
 \Pr\bigl[P_d(1/2)\text{ is normal}\bigr]\leq e^{-2^{d-20}}.
\]

\begin{proof}[Proof of Corollary~\ref{cor:classes-normal}]
Let $X_d$ be the set of full-dimensional subsets of $\{0,1\}^d$.  We first note
that
\[
 |X_d|=(1-o(1))2^{2^d}.
\]
Indeed, a non-full-dimensional subset is contained in the affine span of at most $d$
cube vertices.  The number of affine spans obtained in this way is at most
\[
 \sum_{j=0}^d\binom{2^d}{j}<2^{d^2+d}.
\]
Every proper affine subspace is contained in a proper affine hyperplane.  Such a
hyperplane contains at most $2^{d-1}$ cube vertices: after choosing a coordinate with
nonzero coefficient in a defining equation, the other $d-1$ coordinates determine at
most one value of the chosen coordinate.  Thus the number of non-full-dimensional
subsets is at most
\[
 {2^{d^2+d}2^{2^{d-1}},}
\]
and its proportion among all $2^{2^d}$ subsets is at most
\[
 2^{d^2+d-2^{d-1}}=o(1).
\]

Let $G_d$ be the symmetry group of the cube, so that $|G_d|=2^d d!$, and let
$Y_d\subseteq X_d$ be the set of full-dimensional subsets whose convex hull is
normal.  Since every $G_d$-orbit has at most $|G_d|$ elements, {and writing $N=N(d)$ for the number of normal $0/1$-equivalence classes,}
\[
 N_d\geq\frac{|X_d|}{|G_d|},
 \qquad
 {N\leq |Y_d|.}
\]
Consequently,
\begin{align*}
 {\frac{N}{N_d}}
 &{\leq\frac{|G_d|\,|Y_d|}{|X_d|}}\\
 &{\leq(1+o(1))2^d d!\,
      \Pr\bigl[P_d(1/2)\text{ is normal}\bigr]}\\
 &{\leq(1+o(1))2^d d!\,e^{-2^{d-20}}
 \longrightarrow0.}
\end{align*}
\end{proof}

\section{Open problems}\label{sec:open}

The examples above show that, for an arbitrary a flag unimodular triangulation does not imply Koszulness of the toric ring.  {But polymatroid base polytopes may still behave differently.} See also the open problems in \cite{BCLLM26,HH02,DLFMR26}.

\begin{problem}[Polymatroid base rings]\label{prob:polymatroid}
Let $P$ be a polymatroid base polytope which has a flag unimodular triangulation.
Must $K[P]$ be Koszul? In particular, is the toric ring of the Fano matroid Koszul \cite{BCLLM26}?
\end{problem}

\begin{problem}[Typical triangulable $(0,1)$-polytopes]\label{prob:conditional}
{Let $u_d$, $fu_d$, and $q_d$ denote, respectively, the numbers of $0/1$-equivalence classes of $d$-dimensional $(0,1)$-polytopes that have a unimodular triangulation, a flag unimodular triangulation, and a quadratic triangulation.}
Do
\[
 \frac{fu_d}{u_d}\longrightarrow0
 \qquad\text{and}\qquad
 \frac{q_d}{fu_d}\longrightarrow0?
\]
\end{problem}

The known class ratios $fu_d/u_d$ in dimensions three, four, and five are
\[
 \frac{10}{11},\qquad \frac{212}{303},\qquad
 \frac{153{,}966}{815{,}303},
\]
while the corresponding ratios $q_d/fu_d$ are
\[
 1,\qquad 1,\qquad \frac{153{,}791}{153{,}966}.
\]

Haase, Paffenholz, Piechnik, and Santos asked how often fixed-dimensional lattice polytopes of bounded normalized volume admit unimodular triangulations \cite[Section~1.5]{HPPS21}. The number of unimodular-equivalence classes of $d$-dimensional lattice polytopes has logarithmic growth on the scale $V^{(d-1)/(d+1)}$ \cite{BV92,BY14}. Equivalently, a family of exponential size $e^{\Theta(m)}$ has the  volume scale $m^{(d+1)/(d-1)}$. This suggests the following two problems.

\begin{problem}[{Nonnormal polytopes at the counting scale}]\label{prob:lattice-family}
{For every fixed $d\geq3$, construct constants $c,C>0$ and, for every sufficiently large $m$, at least $e^{cm}$ pairwise unimodularly not equivalent nonnormal $d$-dimensional lattice polytopes of normalized volume at most
\[
 C m^{(d+1)/(d-1)}.
\]}
\end{problem}

\begin{problem}[{Haase--Paffenholz--Piechnik--Santos at the counting scale}]\label{prob:fixed-volume}
{Fix $d\geq3$. For $V>0$, let $u(d,V)$ be the fraction of unimodular-equivalence classes of $d$-dimensional lattice polytopes of normalized volume at most $V$ that admit a unimodular triangulation. Does
\[
 \limsup_{V\to\infty}u(d,V)^{1/V^{(d-1)/(d+1)}}<1?
\]}
\end{problem}

Backman and Liu proved that every matroid
base polytope has a regular unimodular triangulation \cite{BL25}. {Here we show that this result does not extend to \emph{delta-matroids}; see also \cite{EFLS24,FMN18,CSVY23}.} The even delta-matroid
\[
 D_3=([3],\mathcal F_3),\qquad
 \mathcal F_3=\{\varnothing,\{1,2\},\{1,3\},\{2,3\}\},
\]
has the Reeve tetrahedron $\conv\{000,110,101,011\}$ as its delta-matroid
polytope \cite{Ree57}. This simplex is unimodular, and hence normal, in the
affine lattice generated by its vertices, but it does not have the integer decomposition property in
$\Z^3$: indeed,
\[
 \frac12(111)=\frac14(000+110+101+011),
\]
while $111$ is not a sum of two lattice points of the tetrahedron. By
contrast, Proposition~\ref{prop:delta} gives a four-element delta-matroid
whose polytope is spanning and not normal. Deletion and contraction
correspond to coordinate faces, so a $D_3$ minor obstructs $(\cI)$, whereas
a $D_\star$ minor obstructs normality. For the type-$B$ viewpoint and positive results for lattice-path
{delta-matroids}, see \cite{EFLS24,CSVY23}.

\begin{problem}[{delta-matroid polytopes}]\label{prob:deltamatroid}
{Let $\mathfrak D_n$ be the set of labeled delta-matroids on $[n]$. Does}
\[
 \frac{|\{D\in\mathfrak D_n:P(D)\text{ is normal}\}|}
 {|\mathfrak D_n|}\longrightarrow0?
\]
At least, does
\[
 \frac{|\{D\in\mathfrak D_n:P(D)\text{ has a unimodular triangulation}\}|}
 {|\mathfrak D_n|}\longrightarrow0?
\]
\end{problem}

{The asymptotic enumeration of \emph{delta-matroids} was initiated in \cite{FMN18}. In the related $q$-analogue, Degen and K\"uhne proved that asymptotically almost all $q$-matroids are not representable \cite{DK24}.}

\begin{problem}[Koszulness in a family of (0,1)-polytopes]\label{prob:koszul}
Determine which of the remaining $150$ quadratically generated toric rings are Koszul.
Does this family contain a Koszul toric ring with no quadratic Gr\"obner basis?  Find a
criterion in terms of the polytope, a flag unimodular triangulation, or the quadratic
fibers.
\end{problem}

Finally, we ask two minimality questions about the examples found in the classification.

\begin{problem}
Is there a lattice polytope of dimension at most four that has a flag unimodular triangulation but no quadratic triangulation?
\end{problem}

The smallest example found in the census has $13$ vertices and normalized volume $29$.
\begin{problem}
Find a lattice polytope of minimum normalized volume, and one with the minimum number of vertices, that has a flag unimodular triangulation but no quadratic triangulation.
\end{problem}

\section*{\texorpdfstring{Acknowledgments}{Acknowledgments}}
The author is grateful to Jes{\'u}s A. De Loera for extensive advice and many
detailed discussions, and to Christopher Eur for suggesting that {delta-matroid}
polytopes be investigated.  This work was partially supported by NSF grants
DMS-2348578 and DMS-2434665.

\section*{\texorpdfstring{Tool and computational resource disclosure}{Tool and computational resource disclosure}}

The author takes full responsibility for all mathematical statements, references and arguments. The proof of the classification is computer assisted with SAT Solvers. The code will be provided upon request. ChatGPT 5.6 Pro was only used as an auxiliary editorial tool for proofreading and preliminary literature search assistance.

\bibliographystyle{amsalpha}
\bibliography{most_01_polytopes}

\end{document}